\documentclass[12pt]{elsarticle}

 \usepackage{graphics}
\usepackage{amssymb}
\usepackage{amsfonts}
\usepackage{pifont}
\usepackage{amsmath}
\usepackage{latexsym}
\usepackage{bm}
\usepackage{amssymb}
\usepackage{amsthm}
\usepackage{graphics}
\usepackage{epstopdf}
\usepackage{tikz}
\def\bi{\begin{align}}
\def\bin{\begin{align*}}
\newcommand{\ein}{\end{align*}}
\newcommand{\ei}{\end{align}}
\newcommand{\dif}{\mathrm{d}}

\newcommand{\sumli}{\sum\limits}
\newcommand{\ba}{\begin{array}}
\newcommand{\ea}{\end{array}}
\renewcommand{\vec}[1]{\bm{#1}}

\numberwithin{equation}{section} \numberwithin{figure}{section}
\numberwithin{table}{section}

\theoremstyle{definition}
\newtheorem{defi}{Definition}[section]
\theoremstyle{plain}

\theoremstyle{plain}
\newtheorem{thm}{Theorem}[section]

\newtheorem{rem}{Remark}[section]

\begin{document}

\begin{frontmatter}



\title{A note on continuous-stage Runge-Kutta methods}

%\author{}
\author{Wensheng Tang\corref{cor1}}
\ead{tangws@lsec.cc.ac.cn} \cortext[cor1]{Corresponding author.}
\address{College of Mathematics and Statistics,\\
    Changsha University of Science and Technology,\\ Changsha 410114, China}
\address{Hunan Provincial Key Laboratory of Mathematical\\
    Modeling and Analysis in Engineering,\\ Changsha 410114, China}
\author[]{}
%\ead{}

\begin{abstract}

We provide a note on continuous-stage Runge-Kutta methods (csRK) for
solving initial value problems of first-order ordinary differential
equations. Such methods, as an interesting and creative extension of
traditional Runge-Kutta (RK) methods, can give us a new perspective
on RK discretization and it may enlarge the application of RK
approximation theory in modern mathematics and engineering fields. A
highlighted advantage of investigation of csRK methods is that we do
not need to study the tedious solution of multi-variable nonlinear
algebraic equations associated with order conditions. In this note,
we will discuss and promote the recently-developed csRK theory. In
particular, we will place emphasis on geometric integrators
including symplectic methods, symmetric methods and
energy-preserving methods which play a central role in the field of
geometric numerical integration.
\end{abstract}


\end{frontmatter}

% \linenumbers

\section{Introduction}
\label{}



Since the pioneering work of Runge in 1895 \cite{runge95udn} and
Kutta in 1901 \cite{kutta01bzn}, Runge-Kutta (RK) methods have been
developed very well for over a hundred and twenty years
\cite{butcher96aho,butcher87tna,hairernw93sod,hairerw96sod}.
However, continuous-stage Runge-Kutta (csRK) methods, as an
interesting and creative extension of traditional RK methods, begin
entering people's horizons only in recent years. As far as we know,
the most original idea of such methods can be dated back to the
early work by Butcher in 1972 \cite{butcher72ato} (see also
\cite{butcher87tna}), in which RK methods were generalized by
allowing the schemes to be ``continuous" with ``infinitely many
stages". It is surprising that there was a very long period of
quiescence without any development. Until the year 2010, Hairer
pulled the idea back by exploiting it to explain and analyze
energy-preserving collocation methods he proposed in
\cite{hairer10epv}. Subsequently, Tang \& Sun
\cite{Tangs12tfe,Tangsc17dgm} found that some Galerkin
time-discretization methods for ordinary differential equations
(ODEs) can be equivalently transformed into csRK methods, which
implies that RK-type methods bear a close relationship to Galerkin
variational methods. Based on these previous studies, Tang \& Sun
further went deep into the discussion of constructive theory of csRK
methods in \cite{tangs12ana,Tangs14cor}, where orthogonal polynomial
expansion techniques combined with order theory were firstly
utilized. These studies show that an interesting and highlighted
advantage of considering csRK methods is that we do not need to
study the tedious solution of multi-variable nonlinear algebraic
equations associated with order conditions. More recently, Tang et
al have derived some extensions of csRK methods by using similar
techniques, see \cite{Tanglx16cos,Tangz18spc,Tangsz15hos}, in which
continuous-stage partitioned Runge-Kutta methods and
Runge-Kutta-Nystr\"{o}m methods are being proposed and investigated.
Miyatake \& Butcher \cite{miyatake14aep,miyatake15aco} investigate
an energy-preserving condition in terms of the coefficients of csRK
methods for solving Hamiltonian systems, and extend the theory of
exponentially-fitted RK methods in the context of csRK methods.
Besides, Li \& Wu \cite{liw16ffe} proposed functionally fitted
energy-preserving methods for oscillatory nonlinear Hamiltonian
systems and showed that they can be transformed into a class of csRK
methods.

It is well known that geometric numerical integration has become a
major thread in numerical mathematics since around 30 years ago
\cite{Feng95kfc,Fengqq10sga,hairerlw06gni,sanzc94nhp}. RK methods
are greatly developed in such a promising field since 1988
\cite{lasagni88crk,sanz88rkm,suris89ctg}. By introducing a
completely new framework, csRK methods opened up their own important
but distinctive (compared with the traditional RK methods) avenues
in the study of geometric numerical integration. For instance, some
recent literatures show that there exists csRK methods which are
structure-preserving including symplectic csRK methods
\cite{Tangs14cor,tangs12ana}, conjugate-symplectic (up to a finite
order) csRK methods \cite{hairerz13oco,Tangs14cor}, symmetric csRK
methods \cite{hairer10epv,Tangs14cor,tangs12ana}, energy-preserving
csRK methods \cite{brugnanoit10hbv,miyatake15aco,Celledoni09mmoqw,
hairer10epv,quispelm08anc,Tangs14cor,tangs12ana}. Particularly,
there are fruitful energy-preserving methods being proposed from
different approaches recently, e.g., energy-preserving trapezoidal
methods \cite{Iavernarop07sst}, average vector field method (AVFM)
(a kind of discrete gradient method) \cite{quispelm08anc},
Hamiltonian boundary value methods (HBVMs) \cite{brugnanoit10hbv},
continuous time finite element methods (TFEMs)
\cite{betschs00iec,chent07cfe,Tangs12tfe}, energy-preserving
collocation methods (EPCMs) \cite{hairer10epv}. However, all these
methods can be unified in the framework of csRK methods
\cite{Tangs12tfe}. In addition, csRK methods may promote the
investigation of energy-preserving methods which are conjugate
symplectic (up to a finite order) \cite{hairerz13oco,Tangs14cor}.

It is worth mentioning that some special-purpose algorithms are
impossible to exist in the classic context of RK methods but they
can be created fruitfully within the new framework. For example,
Celledoni et al \cite{Celledoni09mmoqw} proved that there exists no
energy-preserving RK methods for general non-polynomial Hamiltonian
systems, but energy-preserving csRK methods obviously exist
\cite{miyatake15aco,hairer10epv,quispelm08anc,Tangs14cor,tangs12ana}.
Furthermore, some numerical integrators can not be perfectly
explained in the classic RK framework, but they can be clearly
understood \cite{Tangs12tfe,Tangsc17dgm} with the help of csRK
methods (e.g., AVFM \cite{quispelm08anc}, $\infty$-HBVMs
\cite{brugnanoit10hbv}, EPCMs \cite{hairer10epv}, Galerkin TFEMs
\cite{betschs00iec,chent07cfe} etc). Hence, it seems that
continuous-stage methods provide us a new realm for numerical
solution of ODEs and it may produce new applications in various
fields especially in geometric numerical integration
\cite{tangs12ana,Tangs14cor,miyatake14aep,miyatake15aco,
Tanglx16cos,Tangz18spc,Tangsz15hos}.

This note is organized as follows. In Section 2, we contrive to
investigate the construction of csRK methods. Based on polynomial
expansion techniques, two effective ways to obtain csRK methods will
be introduced. Section 3 is devoted to discussing the geometric
numerical integration by csRK methods. Some algebraic conditions for
geometric integration are presented, and the ideas of designing
geometric integrators are sketched with the help of them. In the
final section, we give some concluding remarks to end this note.

%%%%%%%%%%%%%%%%%%%%%%%%%%%%%%%%%%%%%%%%%%%%%%%%%%%%%%%%%%%%%%%%%%%%%%%%%%%%%%%%%%%%%%%%%%%%%%%%%
\section{Construction of csRK methods}

For an initial value problem of first-order system in the form
\begin{equation}\label{eq:ode}
\dot{\vec{z}}=\vec{f}(t,\vec{z}),\quad \vec{z}(t_0)=\vec{z}_0\in
\mathbb{R}^d,
\end{equation}
we introduce the following definition of csRK methods.

\begin{defi}\cite{hairer10epv,Tangs14cor}\label{defi1}
Let $A_{\tau,\, \sigma}$ be a function of two variables $\tau$,
$\sigma$ $\in [0, 1]$, and $B_\tau$, $C_\tau$ be functions of
$\tau\in [0, 1]$. The one-step method $\Phi_h: \vec{z}_0 \mapsto
\vec{z}_{1}$ given by
\begin{equation}\label{crk}
\begin{split}
&\vec{Z}_\tau=\vec{z}_0+h\int_0^{1}A_{\tau,\,\sigma}\vec{f}(t_0+C_\sigma
h,\vec{Z}_\sigma)\,\dif
\sigma,\;\tau\in[0,\,1],\\
&\vec{z}_{1}=\vec{z}_0+h\int_0^{1}B_\tau\vec{f}(t_0+C_\tau
h,\vec{Z}_\tau)\,\dif\tau,
\end{split}
\end{equation}
is called a continuous-stage Runge-Kutta (csRK) method, where
$\vec{Z}_\tau\approx \vec{z}(t_0+C_\tau h).$ For the sake of
internal consistency, here we often assume that
\begin{equation}\label{consis}
C_\tau=\int_0^1A_{\tau,\,\sigma}\,\dif\sigma.
\end{equation}
A csRK method is of order $p$, if as $h\rightarrow0$, for all
sufficiently regular problem \eqref{eq:ode} its \emph{local error}
satisfies
\begin{equation*}
\vec{z}_{1}-\vec{z}(t_0+h)=\mathcal{O}(h^{p+1}).
\end{equation*}
\end{defi}
The uniqueness and existence of the solution of csRK schemes are
guaranteed by the following theorem.
\begin{thm}\cite{Tangs12tfe,miyatake15aco}\label{unex}
Assume $\vec{f}$ is Lipschitz continuous with constant $L$.  If step
size $h$ satisfies
$$h<\frac{1}{L\underset{\tau\in[0,1]}{max}\int_0^{1}|A_{\tau,\,\sigma}|\,\dif
\sigma},$$ then there exists a unique solution of \eqref{crk}.
\end{thm}
\begin{thm}\label{prop-csrk}
Assume a csRK method with coefficients
$(A_{\tau,\sigma},\,B_\tau,\,C_\tau)$ satisfies the following two
conditions:
\begin{itemize}
  \item[\emph{(a)}] $\chi(\xi)=\int_0^{\xi}B_\tau\,\dif
  \tau,\;\xi\in[0,1]$ has an inverse function and $B_\tau$ is non-vanishing almost everywhere
   (e.g., if $B_\tau>0$ in $[0,1]$, then this condition is fulfilled);
  \item[\emph{(b)}] $\breve{B}(\rho)$ holds for some integer $\rho\geq1$, where\footnote{This
  condition is always fulfilled for a csRK method of order at least 1 (cf., simplifying
  conditions in \eqref{sim_assu}).}
  \begin{equation*} \breve{B}(\rho):\quad \int_0^1B_\tau
C_\tau^{\kappa-1}\,\dif \tau=\frac{1}{\kappa},\quad
\kappa=1,\ldots,\rho,
\end{equation*}
\end{itemize}
then the method can always be transformed into a new csRK method
with $B_\tau\equiv1,\,\tau\in[0,\,1]$.
\end{thm}
\begin{proof}
Note that the special case with $B_\tau>0$ has been given in
\cite{Tangs14cor}, and the more general case can be proved very
similarly (cf. \cite{Tangs14cor}, Proposition 2.1, page 181).
\end{proof}
In what follows, we attempt to construct csRK methods by using
orthogonal polynomial expansion techniques. For convenience, we
often assume $B_\tau\equiv1$ (in view of Theorem \ref{prop-csrk})
and let $C_\tau\equiv\tau$ for the following discussions. Firstly,
we need to introduce the following $\iota$-degree normalized shifted
Legendre polynomial $P_\iota(x)$ by using the Rodrigues' formula
$$P_0(x)=1,\;P_\iota(x)=\frac{\sqrt{2\iota+1}}{\iota!}
\frac{d^\iota}{dx^\iota}\Big(x^\iota(x-1)^\iota\Big),\;
\;\iota=1,2,3,\cdots.$$ They are orthogonal to each other with
respect to the $L^2$ inner product on $[0,\,1]$
$$\int_0^1 P_\iota(t) P_\kappa(t)\,\dif t= \delta_{\iota\kappa},
\quad\iota,\,\kappa=0,1,2,\cdots,$$ and satisfy the following
integration formulae \cite{hairerw96sod}
\begin{equation}\label{property}
\begin{split} &\int_0^xP_0(t)\,\dif
t=\xi_1P_1(x)+\frac{1}{2}P_0(x), \\
&\int_0^xP_\iota(t)\,\dif
t=\xi_{\iota+1}P_{\iota+1}(x)-\xi_{\iota}P_{\iota-1}(x),\quad
\iota=1,2,3,\cdots,\\
&\int_x^1P_\iota(t)\,\dif
t=\delta_{\iota0}-\int_0^{x}P_\iota(t)\,\dif t,\quad
\iota=0,1,2,\cdots,
\end{split}
\end{equation}
where $\xi_\iota:=\frac{1}{2\sqrt{4\iota^2-1}}$ and
$\delta_{\iota\kappa}$ is the Kronecker symbol.

Since $\{P_i(\tau)P_j(\sigma)\}_{i,j=0}^\infty$ forms a complete
orthogonal basis in function space $L^2([0,1]\times [0,1])$, we can
expand $A_{\tau,\,\sigma}$ as
\begin{equation}\label{expansion}
A_{\tau,\,\sigma}=\sum\limits_{0\leq i,j\in\mathbb{Z}}\alpha_{(i,j)}
P_i(\tau)P_j(\sigma),\quad\alpha_{(i,j)}\in \mathbb{R},
\end{equation}
where $\alpha_{(i,j)}$ are real parameters to be determined.

By substituting \eqref{expansion} into \eqref{consis}, we have
\begin{equation*}
C_\tau=\int_0^1A_{\tau,\,\sigma}\,\dif\sigma=\int_0^1(\sum\limits_{0\leq
i,j\in\mathbb{Z}}\alpha_{(i,j)}P_i(\tau)P_j(\sigma))d\sigma
=\sum\limits_{i\geq0}\alpha_{(i,0)} P_i(\tau).
\end{equation*}
Note that the first formula of \eqref{property} implies
\begin{equation}\label{tau_rew}
C_\tau\equiv\tau=\frac{1}{2}P_0(\tau)+\frac{\sqrt{3}}{6}P_1(\tau),
\end{equation}
then, by comparing the two formulae above with each other, we get
\begin{equation*}
\alpha_{(0,0)}=\frac{1}{2},\;\alpha_{(1,0)}=\frac{\sqrt{3}}{6},\;\alpha_{(i,0)}=0,\;i\geq2.
\end{equation*}

\subsection{Construction of csRK methods order by order}

Analogously to the traditional case of RK-type methods
\cite{hairerlw06gni}, by B-series theory we have the following order
conditions up to order 4 (under the condition \eqref{consis}):
\begin{equation*}
\begin{array}{|lllll|l|}
(1)  \int_0^1 B_\tau  d\tau=1; & & & & &(5) \int_0^1 B_\tau
C_\tau^3 d\tau=\frac{1}{4};\\[5pt]
(2) \int_0^1 B_\tau C_\tau  d\tau=\frac{1}{2};  & & & & &(6)
(\int_0^1)^2
B_\tau C_\tau C_{\sigma} A_{\tau,\sigma}d\tau d\sigma=\frac{1}{8}; \\[5pt]
(3) \int_0^1 B_\tau C_\tau^2 d\tau=\frac{1}{3}; & & & & &(7)
(\int_0^1)^2 B_\tau C_{\sigma}^2
A_{\tau,\sigma}d\tau d\sigma=\frac{1}{12}; \\[5pt]
(4) (\int_0^1)^2 B_\tau C_{\sigma}A_{\tau,\sigma}d\tau
d\sigma=\frac{1}{6}; & & & & & (8) (\int_0^1)^3 B_\tau C_{\rho}
A_{\tau,\sigma}A_{\sigma,\rho} d\tau d\sigma d\rho=\frac{1}{24}.
\end{array}
\end{equation*}

If the condition (1) holds, then the csRK method is of order 1; if
conditions (1)-(2) hold, then the csRK method is of order 2; if
conditions (1)-(4) hold, then the csRK method is of order 3; if
conditions (1)-(8) hold, then the csRK method is of order 4.

By hypothesis (i.e., $B_\tau\equiv1$ and $C_\tau\equiv\tau$),
conditions (1)-(3) and (5) are automatically satisfied. Therefore,
the remaining 4 conditions are to be considered. Our approach is to
substitute the expansion formula \eqref{expansion} into the
conditions one by one so as to get the requirements in terms of the
expansion coefficients. In the following the orthogonality of
Legendre polynomials and formula \eqref{tau_rew} will be used
several times.

For condition (4):
\begin{equation*}
\begin{split}
&\big(\int_0^1\big)^2B_\tau A_{\tau,\sigma}C_{\sigma}d\tau d\sigma
=\int_0^1\big(\int_0^1A_{\tau,\sigma}d\tau\big)\sigma
d\sigma\\
&=\int_0^1(\sum\limits_{j\geq0}\alpha_{(0,j)}
P_j(\sigma))\big(\frac{1}{2}P_0(\sigma)+\frac{\sqrt{3}}{6}P_1(\sigma)\big)
d\sigma=\frac{1}{2}\alpha_{(0,0)}+\frac{\sqrt{3}}{6}\alpha_{(0,1)}=\frac{1}{6},
\end{split}
\end{equation*}
which then gives $\alpha_{(0,1)}=-\frac{\sqrt{3}}{6}$.

For condition (6):
\begin{equation*}
\begin{split}
&(\int_0^1)^2 B_\tau C_\tau C_{\sigma} A_{\tau,\sigma}d\tau
d\sigma=\int_0^1(\int_0^1\tau A_{\tau,\sigma}d\tau)\sigma
d\sigma\\
&=\int_0^1\Big(\int_0^1(\frac{1}{2}P_0(\tau)+\frac{\sqrt{3}}{6}P_1(\tau))
(\sum\limits_{0\leq
i,j\in\mathbb{Z}}\alpha_{(i,j)}
P_i(\tau)P_j(\sigma))d\tau\Big)\sigma
d\sigma\\
&=\int_0^1\big(\frac{1}{2}\sum\limits_{j\geq0}\alpha_{(0,j)}P_j(\sigma)+
\frac{\sqrt{3}}{6}\sum\limits_{j\geq0}\alpha_{(1,j)}P_j(\sigma)\big)
(\frac{1}{2}P_0(\sigma)+\frac{\sqrt{3}}{6}P_1(\sigma))
d\sigma\\
&=\frac{1}{4}\alpha_{(0,0)}+\frac{\sqrt{3}}{12}\alpha_{(1,0)}+
\frac{\sqrt{3}}{12}\alpha_{(0,1)}+\frac{1}{12}\alpha_{(1,1)}=\frac{1}{8},
\end{split}
\end{equation*}
which then gives $\alpha_{(1,1)}=0$.

For condition (7):
\begin{equation*}
\begin{split}
&(\int_0^1)^2 B_\tau C_{\sigma}^2 A_{\tau,\sigma}d\tau
d\sigma=\int_0^1(\int_0^1 A_{\tau,\sigma}d\tau)\sigma^2
d\sigma\\
&=\int_0^1\big(\sum\limits_{j\geq0}\alpha_{(0,j)}
P_j(\sigma)\big)(\frac{1}{3}P_0(\sigma)+\frac{\sqrt{3}}{6}P_1(\sigma)
+\frac{\sqrt{5}}{30}P_2(\sigma))d\sigma\\
&=\frac{1}{3}\alpha_{(0,0)}+\frac{\sqrt{3}}{6}\alpha_{(0,1)}+
\frac{\sqrt{5}}{30}\alpha_{(0,2)}=\frac{1}{12},
\end{split}
\end{equation*}
which then gives $\alpha_{(0,2)}=0$. Here we used an identity
$$\sigma^2=2\int_0^{\sigma}(\int_0^xP_0(t)\,\dif
t)\,\dif x=\frac{1}{3}P_0(\sigma)+\frac{\sqrt{3}}{6}P_1(\sigma)
+\frac{\sqrt{5}}{30}P_2(\sigma)$$ which is deduced from
\eqref{property}.

For condition (8):
\begin{equation*}
\begin{split}
&(\int_0^1)^3 B_\tau C_{\rho} A_{\tau,\sigma}A_{\sigma,\rho} d\tau
d\sigma d\rho=\int_0^1(\int_0^1 A_{\tau,\sigma}d\tau)(\int_0^1 \rho
A_{\sigma,\rho}d\rho)d\sigma\\
&=\int_0^1\big(\sum\limits_{j\geq0}\alpha_{(0,j)}
P_j(\sigma)\big)\big(\frac{1}{2}\sum\limits_{i\geq0}\alpha_{(i,0)}P_i(\sigma)+
\frac{\sqrt{3}}{6}\sum\limits_{i\geq0}\alpha_{(i,1)}P_i(\sigma)\big)d\sigma\\
&=\frac{1}{2}\sum\limits_{i\geq0}\alpha_{(0,i)}\alpha_{(i,0)}+\frac{\sqrt{3}}{6}
\sum\limits_{i\geq0}\alpha_{(0,i)}\alpha_{(i,1)}=\frac{1}{24},
\end{split}
\end{equation*}
which then gives
$\frac{1}{2}\sum\limits_{i\geq2}\alpha_{(0,i)}\alpha_{(i,0)}+\frac{\sqrt{3}}{6}
\sum\limits_{i\geq2}\alpha_{(0,i)}\alpha_{(i,1)}=0$. Take into
account that $\alpha_{(0,2)}=0$ and $\alpha_{(i,0)}=0,\;i\geq2$, and
then it ends up with
$\sum\limits_{i\geq3}\alpha_{(0,i)}\alpha_{(i,1)}=0$.

\begin{thm}\label{orderbyorder}
Under the assumptions $B_\tau\equiv1$ and $C_\tau\equiv\tau$, the
csRK method \eqref{crk} with $A_{\tau, \sigma}$ given by
\begin{equation}\label{expansion2}
A_{\tau,\,\sigma}=\frac{1}{2}+\frac{\sqrt{3}}{6}P_1(\tau)+
\sum\limits_{i\geq0,j\geq1}\alpha_{(i,j)}
P_i(\tau)P_j(\sigma),\quad\alpha_{(i,j)}\in \mathbb{R},
\end{equation}
is of order $2$ at least. Moreover, if we additionally require
$\alpha_{(0,1)}=-\frac{\sqrt{3}}{6}$, then the method is of order
$3$ at least; if we require, additionally,
\begin{equation}\label{coeorder4m}
\alpha_{(1,1)}=0,\;\alpha_{(0,2)}=0,\;
\sum\limits_{i\geq3}\alpha_{(0,i)}\alpha_{(i,1)}=0,
\end{equation}
then the method is of order $4$ at least.
\end{thm}

\subsection{Construction of high-order csRK methods}

Although we can construct csRK methods of arbitrarily high order via
the above technique order by order, it is not an easy task to derive
higher order methods, seeing that the number of order conditions
will increase dramatically
\cite{hairerlw06gni,hairernw93sod,hairerw96sod} and one has to
conduct more tedious and complicated computation. To overcome these
difficulties, we have to use the following \emph{simplifying
assumptions} \cite{hairer10epv}
\begin{equation}\label{sim_assu}
\begin{split}
&\breve{B}(\rho):\quad \int_0^1B_\tau C_\tau^{\kappa-1}\,\dif
\tau=\frac{1}{\kappa},\quad \kappa=1,\ldots,\rho,\\
&\breve{C}(\eta):\quad
\int_0^1A_{\tau,\,\sigma}C_\sigma^{\kappa-1}\,\dif
\sigma=\frac{1}{\kappa}C_\tau^{\kappa},\quad \kappa=1,\ldots,\eta,\\
&\breve{D}(\zeta):\quad \int_0^1B_\tau C_\tau^{\kappa-1}
A_{\tau,\,\sigma}\,\dif
\tau=\frac{1}{\kappa}B_\sigma(1-C_\sigma^{\kappa}),\quad
\kappa=1,\ldots,\zeta.
\end{split}
\end{equation}
Actually, Tang \& Sun \cite{Tangs14cor} have investigated the
construction of high-order methods by using these \emph{simplifying
assumptions}. Now we give a brief review of some existing results.
The following result is completely similar to the classic result by
Butcher in 1964 \cite{butcher64ipa}.
\begin{thm}\cite{Tangs14cor}\label{crk:order}
If the coefficients $(A_{\tau,\,\sigma},\,B_\tau,\,C_\tau)$ of
method \eqref{crk} satisfy $\breve{B}(\rho)$, $\breve{C}(\alpha)$
and $\breve{D}(\beta)$, then the method is of order at least $\min
(\rho,2\alpha+2,\,\alpha+\beta+1)$.
\end{thm}
\begin{thm}\cite{Tangs14cor} \label{mainthm}
For a csRK method with $B_\tau\equiv1$ and $C_\tau=\tau$ (then
$\breve{B}(\infty)$ holds), the following two statements are
equivalent to each other:\emph{(a)} Both $\breve{C}(\alpha)$ and
$\breve{D}(\beta)$ hold; \emph{(b)} The coefficient
$A_{\tau,\,\sigma}$ has the following form in terms of Legendre
polynomials
\begin{equation}\label{coef}
A_{\tau,\,\sigma}=\frac{1}{2}+\sum_{\iota=0}^{N_1}\xi_{\iota+1}
P_{\iota+1}(\tau)P_\iota(\sigma)-\sum_{\iota=0}^{N_2}\xi_{\iota+1}
P_{\iota+1}(\sigma)P_\iota(\tau)+\sum_{i\geq\beta,\,
j\geq\alpha}\alpha_{(i,j)}P_i(\tau)P_j(\sigma),
\end{equation}
where $N_1=\max(\alpha-1,\,\beta-2)$,
$N_2=\max(\alpha-2,\,\beta-1)$,
$\xi_\iota=\frac{1}{2\sqrt{4\iota^2-1}}$ and $\alpha_{(i,j)}$ are
any real parameters.
\end{thm}

By combining Theorem \ref{crk:order} with Theorem \ref{mainthm} we
can easily construct csRK methods of arbitrarily high order, the
order of which are given by $\min
(\infty,\,2\alpha+2,\,\alpha+\beta+1)=\min
(2\alpha+2,\,\alpha+\beta+1)$. For example, if we take
$\alpha=2,\,\beta=1$ in Theorem \ref{mainthm}, then we get a family
of 4-order methods which can be retrieved by taking
\begin{equation*}
\alpha_{(0,i)}=\alpha_{(i,1)}=0,\;i\geq3,\;\alpha_{(2,1)}=\xi_2=\frac{\sqrt{15}}{30},
\end{equation*}
in Theorem \ref{orderbyorder}. Note that methods constructed by
Theorem \ref{orderbyorder} cover all the methods given by Theorem
\ref{mainthm} (as we construct methods up to order 4). This implies
that we will lose the opportunity to discover many other new csRK
methods by using Theorem \ref{mainthm}, even though it is much
easier to construct high-order csRK methods compared with the
approach shown in subsection 2.1.

To derive a practical csRK method, we need to get a finite form of
$A_{\tau,\,\sigma}$ by truncating the series \eqref{coef}. In such a
case, without loss a generality, we assume $A_{\tau,\,\sigma}$ is a
bivariate polynomial of degree $\pi_A^{\tau}$ in $\tau$ and degree
$\pi_A^{\sigma}$ in $\sigma$. Applying a quadrature formula
$(b_i,c_i)_{i=1}^s (0\leq c_i\leq1)$ to (\ref{crk}), we derive an
$s$-stage RK method
\begin{equation}\label{crk:quad}
\begin{split}
&\widetilde{\vec{Z}}_i=\vec{z}_0+h\sum_{j=1}^sb_jA_{c_i,\,c_j}\vec{f}(t_0+c_{j}h,
\widetilde{\vec{Z}}_j),\quad i=1,\cdots,s,\\
&\vec{z}_{1}=\vec{z}_0+h\sum_{i=1}^sb_{i}B_{c_i}\vec{f}(t_0+c_{i}h,\widetilde{\vec{Z}}_i),
\end{split}
\end{equation}
where $\widetilde{\vec{Z}}_i\approx\vec{Z}_{c_i}$.
\begin{thm}\cite{Tanglx16cos}\label{qua:csRK}
Assume $A_{\tau,\,\sigma}$ is a bivariate polynomial of degree
$\pi_A^{\tau}$ in $\tau$ and degree $\pi_A^{\sigma}$ in $\sigma$,
and the quadrature formula $(b_i,c_i)_{i=1}^s$ is of
order\footnote{The quadrature formula is of order $p$ iff $\int_0^1
f(x)\, \dif x=\sumli_{i=1}^s b_i f(c_i)$ holds for any polynomial
$f(x)$ of degree up to $p-1$.} $p$. If a csRK method \eqref{crk}
with coefficients $(A_{\tau,\,\sigma},\,B_\tau,\,C_\tau)$ satisfies
$B_\tau\equiv1 ,\,C_\tau=\tau$ (then $\breve{B}(\infty)$ holds) and
both $\breve{C}(\eta)$, $\breve{D}(\zeta)$ hold, then the classic RK
method \eqref{crk:quad} with coefficients
$(b_{j}A_{c_i,c_j},b_i,\,c_i)$ is of order at least
$$\min(p, 2\alpha+2, \alpha+\beta+1),$$
where $\alpha=\min(\eta, p-\pi_A^{\sigma})$ and $\beta=\min(\zeta,
p-\pi_A^{\tau})$.
\end{thm}

Theorem \ref{qua:csRK} tells us how to construct a traditional RK
scheme based on csRK methods. The most highlighted advantage of such
approach to construct RK-type methods is that we do not need to
consider and study the tedious solution of nonlinear algebraic
equations deduced from order conditions. It turns out that this
approach \cite{tangs12ana,Tangs14cor} is comparable to the
W-transformation technique proposed by Hairer \& Wanner
\cite{hairerw96sod}.

%%%%%%%%%%%%%%%%%%%%%%%%%%%%%%%%%%%%%%%%%%%%%%%%%%%%%%%%%%%%%%%%%%%%%%%%%%%%%%%%%%%%%%%%%%%%%%%%%
\section{Geometric numerical integration by csRK methods}

In this section, we mainly focus on the geometric numerical
integration of Hamiltonian problem
\begin{equation}\label{Hs}
\dot{\vec{z}}=J^{-1}\nabla H(\vec{z}), \quad
\vec{z}(t_0)=\vec{z}_0\in \mathbb{R}^{2d},
\end{equation}
where $J=\begin{pmatrix}\vec{0}&I_d\\-I_d&\vec{0}\end{pmatrix}$
(with $I_d$ the $d\times d$ identity matrix) is a standard structure
matrix,  $H:\mathbb{R}^{2d}\rightarrow\mathbb{R}$ is the Hamiltonian
function which generally represents the total energy of the given
system. The system \eqref{Hs} has two main geometric properties
\cite{Arnold89mmo}:
\begin{itemize}
  \item[(a)] Energy preservation:\;
$H(\vec{z}(t))\equiv H(\vec{z}(t_0))$ for $\forall t$;
  \item[(b)] Symplecticity (Poincar\'{e} 1899):\;
$\dif\vec{z}(t)\wedge J\dif\vec{z}(t)=\dif\vec{z}(t_{0})\wedge
J\dif\vec{z}(t_{0})$ for $\forall t$.
\end{itemize}
It is known that property (b) is a characteristic property for
Hamiltonian systems (see \cite{hairerlw06gni}, Theorem 2.6, page
185) and it essentially implies (a). A well-known negative result
given by Ge \& Marsden \cite{gem88lph} manifests that, generally, we
can not have a numerical method which exactly preserves both
properties at the same time\footnote{For linear Hamiltonian systems,
there exists numerical methods which exactly preserve energy and
symplecticity simultaneously, e.g., symplectic RK methods can
preserve all quadratic invariants including the quadratic
Hamiltonian \cite{hairerlw06gni}.}. It has been evidenced that
symplectic methods possess a nearly energy-preserving property
(exactly preserve a modified Hamiltonian) for long-term computation
\cite{hairerlw06gni}, while energy-preserving methods will loss the
symplecticity in general---It may possibly leads to incorrect phase
space behavior. Particularly, when symplectic methods are applied to
integrable and near-integrable systems, they produce excellent
numerical behaviors: linear error growth, long-time
near-conservation of first integrals, existence of invariant tori
\cite{hairerlw06gni,Shang99kam}. For these reasons, symplectic
methods have been drawn more attentions in geometric integration.
However, energy-preserving methods are also of interest in many
fields, e.g., molecular dynamics, plasma physics etc
\cite{Fengqq10sga,hairerlw06gni,sanzc94nhp}. An interesting result
is that there exists an energy-preserving B-series integrator which
is conjugate to a symplectic method \cite{Chartierfm06aaa}, but it
remains a challenge to construct a computational method owning such
a symplectic-like property \cite{hairerz13oco}. Besides, symmetric
methods are popular for solving many time-reversible problems
arising in various fields, and they share many similar excellent
long-time properties with symplectic methods especially when they
are applied to (near-)integrable systems \cite{hairerlw06gni}. In
general, energy-preserving methods for time-reversible Hamiltonian
system are often expected to be symmetric.

\subsection{Symplectic csRK methods}

In this part, we will firstly study the condition for csRK methods
to be symplectic, and then discuss the construction of symplectic
methods.

\begin{thm}\label{symcond1}
If the coefficients of a csRK method \eqref{crk} satisfy
\begin{equation}\label{symcond}
B_\tau A_{\tau,\sigma}+B_\sigma A_{\sigma,\tau}\equiv B_\tau
B_\sigma,\;\; \tau,\,\sigma \in [0,1],
\end{equation}
then it is symplectic.
\end{thm}
\begin{proof}
Applying a csRK method to Hamiltonian system \eqref{Hs} it gives
\begin{equation}\label{crkHs}
\begin{split}
&\vec{Z}_\tau=\vec{z}_0+h\int_0^{1}A_{\tau,\,\sigma}\vec{f}(\vec{Z}_\sigma)\,\dif
\sigma,\;\tau\in[0,\,1],\\
&\vec{z}_{1}=\vec{z}_0+h\int_0^{1}B_\tau\vec{f}(\vec{Z}_\tau)\,\dif\tau,
\end{split}
\end{equation}
where $\vec{f}(\vec{z})=J^{-1}\nabla H(\vec{z})$. Our aim is to
verify the following identity
\begin{equation}\label{veri}
\dif\vec{z}_{1}\wedge J\dif\vec{z}_{1}=\dif\vec{z}_{0}\wedge
J\dif\vec{z}_{0}.
\end{equation}

In the following, we denote the $(k,l)$-element of $J$ by $J_{kl}$
and the $i$th component of a vector $\vec{v}$ by $\vec{v}^{(i)}$.
From the first formula of \eqref{crkHs}, we conclude
\begin{equation}\label{ab}
\dif\vec{z}^{(i)}_0=\dif\vec{Z}_\tau^{(i)}-h\int_0^{1}A_{\tau,\,\sigma}
\dif\vec{f}^{(i)}(\vec{Z}_\sigma)\,\dif\sigma
=\dif\vec{Z}_\sigma^{(i)}-h\int_0^{1}A_{\sigma,\,\tau}\dif\vec{f}^{(i)}
(\vec{Z}_\tau)\,\dif\tau,\;1\leq
i\leq 2d,
\end{equation}
which will be used later. Making difference between left-hand side
and right-hand side of \eqref{veri}, it yields
\begin{equation*}
\begin{split}
&\qquad \dif\vec{z}_{1}\wedge J\dif\vec{z}_{1}-\dif\vec{z}_{0}\wedge
J\dif\vec{z}_{0}=\sum_{k,l}^{2d} J_{kl}\dif\vec{z}^{(k)}_{1}\wedge
\dif\vec{z}^{(l)}_{1}-\sum_{k,l}^{2d}
J_{kl}\dif\vec{z}^{(k)}_{0}\wedge
\dif\vec{z}^{(l)}_{0}\\
&=\sum_{k,l}^{2d}
J_{kl}\Big((\dif\vec{z}^{(k)}_0+h\int_0^{1}B_\tau\dif\vec{f}^{(k)}(\vec{Z}_\tau)\,\dif\tau)\wedge
(\dif\vec{z}^{(l)}_0+h\int_0^{1}B_\sigma\dif\vec{f}^{(l)}(\vec{Z}_\sigma)\,\dif\sigma)
-\dif\vec{z}^{(k)}_{0}\wedge
\dif\vec{z}^{(l)}_{0}\Big)\\
&=\sum_{k,l}^{2d}
J_{kl}\Big(h\int_0^{1}B_\sigma\dif\vec{z}^{(k)}_0\wedge\dif\vec{f}^{(l)}(\vec{Z}_\sigma)\,\dif\sigma
+h\int_0^{1}B_\tau\dif\vec{f}^{(k)}(\vec{Z}_\tau)\wedge
\dif\vec{z}^{(l)}_0\,\dif\tau\\
&\quad+h^2\int_0^{1}\int_0^{1}B_\tau B_\sigma\dif\vec{f}^{(k)}(\vec{Z}_\tau)\wedge
\dif\vec{f}^{(l)}(\vec{Z}_\sigma)\,\dif\tau\dif\sigma\Big)\\
&=\sum_{k,l}^{2d}J_{kl}\Big(h\int_0^{1}B_\sigma\underbrace{\big(\dif\vec{Z}_\sigma^{(k)}-
h\int_0^{1}A_{\sigma,\,\tau}\dif\vec{f}^{(k)}(\vec{Z}_\tau)\,\dif\tau\big)}_{(a)}
\wedge\dif\vec{f}^{(l)}(\vec{Z}_\sigma)\,\dif\sigma\\
&\quad+h\int_0^{1}B_\tau\dif\vec{f}^{(k)}(\vec{Z}_\tau)\wedge
\underbrace{\big(\dif\vec{Z}_\tau^{(l)}-h\int_0^{1}A_{\tau,\,\sigma}
\dif\vec{f}^{(l)}(\vec{Z}_\sigma)\,\dif\sigma\big)}_{(b)}\,\dif\tau\\
&\quad+h^2\int_0^{1}\int_0^{1}B_\tau B_\sigma\dif\vec{f}^{(k)}(\vec{Z}_\tau)
\wedge\dif\vec{f}^{(l)}(\vec{Z}_\sigma)\,\dif\tau\dif\sigma\Big)\\
&=\sum_{k,l}^{2d}J_{kl}\Big(h\int_0^{1}B_\sigma\dif\vec{Z}_\sigma^{(k)}
\wedge\dif\vec{f}^{(l)}(\vec{Z}_\sigma)\,\dif\sigma+h\int_0^{1}B_\tau
\dif\vec{f}^{(k)}(\vec{Z}_\tau)\wedge
\dif\vec{Z}_\tau^{(l)}\,\dif\tau\\
&\quad-h^2\underbrace{\int_0^{1}\int_0^{1}M_{\tau,\sigma}
\dif\vec{f}^{(k)}(\vec{Z}_\tau)\wedge\dif\vec{f}^{(l)}(\vec{Z}_\sigma)\,
\dif\tau\dif\sigma}_{(c)}\Big)\\
&=h\sum_{k,l}^{2d}J_{kl}\Big(\int_0^{1}B_\sigma\dif\vec{Z}_\sigma^{(k)}
\wedge\dif\vec{f}^{(l)}(\vec{Z}_\sigma)\,\dif\sigma+\int_0^{1}B_\tau
\dif\vec{f}^{(k)}(\vec{Z}_\tau)\wedge
\dif\vec{Z}_\tau^{(l)}\,\dif\tau\Big)\\
&=h\Big(\int_0^{1}B_\sigma \dif\vec{Z}_\sigma \wedge
J\dif\vec{f}(\vec{Z}_\sigma)\,\dif\sigma+\int_0^{1}B_\tau
\dif\vec{f}(\vec{Z}_\tau)\wedge J\dif\vec{Z}_\tau\,\dif\tau\Big)\\
&=2h\int_0^{1}B_\sigma \dif\vec{Z}_\sigma \wedge
J\dif\vec{f}(\vec{Z}_\sigma)\,\dif\sigma,
\end{split}
\end{equation*}
where $(a)$ and $(b)$ are derived by using \eqref{ab}, and $(c)$
vanishes by \eqref{symcond} since $M_{\tau,\sigma}:=B_\tau
A_{\tau,\sigma}+B_\sigma A_{\sigma,\tau}-B_\tau B_\sigma\equiv0$. At
last, the proof finishes by taking into account that
\begin{equation*}
\dif\vec{Z}_\sigma \wedge
J\dif\vec{f}(\vec{Z}_\sigma)=\dif\vec{Z}_\sigma \wedge
JJ^{-1}\nabla^{2}H(\vec{Z}_\sigma)\dif\vec{Z}_\sigma=\dif\vec{Z}_\sigma
\wedge \nabla^{2}H(\vec{Z}_\sigma)\dif\vec{Z}_\sigma=0.
\end{equation*}
\end{proof}
\begin{rem}
The symplectic condition \eqref{symcond} is very similar to the
classic result for traditional RK methods (which has been proved to
be necessary for irreducible methods \cite{hairerlw06gni}). Thus, we
conjecture that the condition is also essentially necessary. We
leave the proof of this conjecture to our future work.
\end{rem}

It is not an easy task to find out all the symplectic csRK methods
from the condition given in Theorem \ref{symcond1}. Tang et al
\cite{Tangs14cor,Tanglx16cos} have presented an alternative
condition for symplecticity, which can be seen as a reduction of
\eqref{symcond}. Now we revisit the result given in
\cite{Tangs14cor,Tanglx16cos}, and actually it suffices for us to
get symplectic integrators of arbitrarily high order.

\begin{thm}\cite{Tangs14cor,Tanglx16cos}\label{symthm}
A csRK method with $B_\tau=1,\,C_\tau=\tau$ is symplectic if
$A_{\tau,\,\sigma}$ has the following form in terms of Legendre
polynomials
\begin{equation}\label{sym:csRK2}
A_{\tau,\,\sigma}=\frac{1}{2}+\sum_{0<i+j\in
\mathbb{Z}}\alpha_{(i,j)}
P_i(\tau)P_j(\sigma),\quad\alpha_{(i,j)}\in \mathbb{R},
\end{equation}
where $\alpha_{(i,j)}$ is skew-symmetric, i.e.,
$\alpha_{(i,j)}=-\alpha_{(j,i)},\,i+j>0$.
\end{thm}
\begin{proof}
Under the assumption $B_\tau=1,\,C_\tau=\tau$, symplectic condition
\eqref{symcond} is reduced to
\begin{equation}\label{symcond2}
A_{\tau,\sigma}+A_{\sigma,\tau}\equiv 1,\; \text{for}\;\;
\tau,\,\sigma \in [0,1],
\end{equation}
By using the expansion \eqref{expansion} and exchanging
$\tau\leftrightarrow\sigma$, we have
\begin{equation*}
A_{\sigma,\,\tau}=\sum_{0\leq i,j\in\mathbb{Z}}\alpha_{(i,j)}
P_i(\sigma)P_j(\tau) =\sum_{0\leq i,j\in\mathbb{Z}}\alpha_{(j,i)}
P_j(\sigma)P_i(\tau).
\end{equation*}
Substituting this formula into \eqref{symcond2} and collecting the
like terms gives
$$\alpha_{(0,0)}=\frac{1}{2};\;\,\alpha_{(i,j)}=-\alpha_{(j,i)},\,i+j>0,$$
which completes the proof.
\end{proof}

Consequently, a simple way to design symplectic csRK methods of
arbitrarily high order pops out by putting Theorem \ref{symthm} and
\ref{mainthm} together, due to that suitable RK coefficients can be
easily tuned according to these theorems. Another way is to
substitute \eqref{sym:csRK2} into order conditions (cf. subsection
2.1) one by one, which then produces symplectic methods order by
order.

Here we give the following result to show that symplectic RK methods
can be easily derived based on symplectic csRK methods. It was shown
in \cite{Tangs14cor,Tanglx16cos} that many classic high-order
symplectic RK methods including Gauss-Legendre RK schemes, Radau IB,
Radau IIB and Lobatto IIIE can be retrieved in this way.

\begin{thm}
The RK scheme \eqref{crk:quad} (with coefficients
$(b_{j}A_{c_i,c_j},b_iB_{c_i},\,c_i)$) based on a symplectic csRK
method with coefficients satisfying \eqref{symcond} is always
symplectic.
\end{thm}
\begin{proof}
By taking into account that
$$B_{c_i}A_{c_i,\,c_j}+B_{c_j}A_{c_j,\,c_i}=B_{c_i}B_{c_j},\;\;i,j=1,\cdots,s,$$
we have
$$(b_iB_{c_i})(b_jA_{c_i,\,c_j})+(b_jB_{c_j})(b_iA_{c_j,\,c_i})=
(b_iB_{c_i})(b_jB_{c_j}),\;\;i,j=1,\cdots,s,$$ which get the final
result by a classic theorem (cf., \cite{hairerlw06gni}, page 192).
\end{proof}

\subsection{Symmetric csRK methods}

As pointed out in \cite{hairerlw06gni}, symmetric methods as well as
symplectic methods play a central role in the geometric integration
of differential equations. In this part, we will give the condition
for a csRK method to be symmetric and then show a simple way to
construct such geometric integrators.
\begin{defi}\cite{hairerlw06gni}
A one-step method $\phi_h$ is called symmetric (or time-reversible)
if it satisfies
$$\phi^*_h=\phi_h,$$
where $\phi^*_h=\phi^{-1}_{-h}$ is referred to as the adjoint method
of $\phi_h$.
\end{defi}
By the definition, a method $z_1=\phi_h(z_0; t_0,t_1)$ is symmetric
if exchanging $h\leftrightarrow -h$, $z_0\leftrightarrow z_1$ and
$t_0\leftrightarrow t_1$ leaves the original method unaltered. From
the definition above, we can prove the following theorem.
\begin{thm}\label{symcon}
Under the assumption \eqref{consis} and we suppose $\breve{B}(\rho)$
holds with $\rho\geq1$ (which means the method is of order at least
1), then a csRK method is symmetric if
\begin{equation}\label{symm}
A_{\tau,\,\sigma}+A_{1-\tau,\,1-\sigma}\equiv
B_{\sigma},\;\;\tau,\,\sigma\in[0,1].
\end{equation}
\end{thm}
\begin{proof}
Obviously, \eqref{symm} implies $B_{\sigma}\equiv B_{1-\sigma}$ in
$[0,1]$. Furthermore, by taking an integral on both sides of
\eqref{symm} with respect to $\sigma$, we get
$C_\tau+C_{1-\tau}\equiv1,\;\;\tau\in[0,1]$.

Next, let us establish the adjoint method of a given csRK method.
From \eqref{crk}, by interchanging $t_0, \vec{z}_0, h$ with $t_1,
\vec{z}_1, -h$ respectively, we have
\begin{equation*}
\begin{split}
&\vec{Z}_\tau=\vec{z}_1-h\int_0^{1}A_{\tau,\,\sigma}\vec{f}(t_1-C_\sigma
h,\vec{Z}_\sigma)\,\dif
\sigma,\;\tau\in[0,\,1],\\
&\vec{z}_{0}=\vec{z}_1-h\int_0^{1}B_\tau\vec{f}(t_1-C_\tau
h,\vec{Z}_\tau)\,\dif\tau,
\end{split}
\end{equation*}
Note that $t_1-C_\tau h=t_0+(1-C_\tau)h$, then the second formula
can be recast as
\begin{equation*}
\vec{z}_{1}=\vec{z}_0+h\int_0^{1}B_\tau\vec{f}(t_0+(1-C_\tau)
h,\vec{Z}_\tau)\,\dif\tau.
\end{equation*}
By plugging it into the first formula, then it ends up with
\begin{equation*}
\begin{split}
&\vec{Z}_\tau=\vec{z}_0+h\int_0^{1}(B_\sigma-A_{\tau,\,\sigma})\vec{f}(t_0+(1-C_\sigma)
h,\vec{Z}_\sigma)\,\dif
\sigma,\;\tau\in[0,\,1],\\
&\vec{z}_{1}=\vec{z}_0+h\int_0^{1}B_\tau\vec{f}(t_0+(1-C_\tau)
h,\vec{Z}_\tau)\,\dif\tau,
\end{split}
\end{equation*}
By replacing $\tau$ and $\sigma$ with $1-\tau$ and $1-\sigma$
respectively, and with the help of change of integral variables, we
obtain an equivalent scheme
\begin{equation*}
\begin{split}
&\vec{Z}^*_\tau=\vec{z}_0+h\int_0^{1}A^*_{\tau,\,\sigma}\vec{f}(t_0+C^*_\sigma
h,\vec{Z}^*_\sigma)\,\dif
\sigma,\;\tau\in[0,\,1],\\
&\vec{z}_{1}=\vec{z}_0+h\int_0^{1}B^*_\tau\vec{f}(t_0+C^*_\tau
h,\vec{Z}^*_\tau)\,\dif\tau,
\end{split}
\end{equation*}
which is the adjoint method of the original method, where
$\vec{Z}^*_\tau=\vec{Z}_{1-\tau}$ and
\begin{equation*}
\begin{split}
&A^*_{\tau,\,\sigma}=B_{1-\sigma}-A_{1-\tau,\,1-\sigma}\equiv B_{\sigma}-A_{1-\tau,\,1-\sigma},\\
&B^*_\tau=B_{1-\tau}\equiv B_\tau,\;\;C^*_\tau=1-C_{1-\tau}\equiv
C_\tau.
\end{split}
\end{equation*}
Note that a csRK method can be uniquely determined by its
coefficients (cf. Theorem \ref{unex}), hence if we require
$A^*_{\tau,\,\sigma}=A_{\tau,\,\sigma}$, i.e., \eqref{symm}, then
the original csRK method is symmetric.
\end{proof}
\begin{rem}
The symmetric condition \eqref{symm} is very similar to the classic
result for traditional RK methods (which has been proved to be
necessary for irreducible methods \cite{hairerlw06gni}). Thus, we
conjecture that the condition is also essentially necessary. We
don't plan to pursue this conjecture here.
\end{rem}
\begin{thm}\label{symm_quad}
If the underlying symmetric csRK method with coefficients
$(A_{\tau,\sigma},B_\tau,C_\tau)$ satisfying the condition of
Theorem \ref{symcon}, then the associated RK method \eqref{crk:quad}
is symmetric, provided that the quadrature weights and abscissae
satisfy $b_{s+1-i}=b_i$ and $c_{s+1-i}=1-c_i$ for all $i$.
\end{thm}
\begin{proof}
An available classic result for an $s$-stage standard RK method
$(a_{ij},\,b_i,\,c_i)$ to be symmetric has revealed the following
sufficient condition (see, e.g., \cite{hairerlw06gni})
\begin{equation*}
a_{ij}+a_{s+1-i,s+1-j}=b_j,\;\;i,\,j=1,\cdots,s.
\end{equation*}
Observe that
\begin{equation}
A_{c_i,\,c_j}+A_{1-c_i,\,1-c_j}= B_{c_j},\;\;i,\,j=1,\cdots,s,
\end{equation}
and on account of $b_{s+1-i}=b_i$, $c_{s+1-i}=1-c_i$, it yields
\begin{equation}
(b_jA_{c_i,\,c_j})+(b_{s+1-j}A_{c_{s+1-i},\,c_{s+1-j}})=
b_jB_{c_j},\;\;i,\,j=1,\cdots,s,
\end{equation}
which completes the proof by the classic result.
\end{proof}
\begin{thm}\cite{Tangs14cor}\label{symmthm}
The csRK method with $B_\tau=1$ and $C_\tau=\tau$ is symmetric if
$A_{\tau,\,\sigma}$ has the following form in terms of Legendre
polynomials
\begin{equation}\label{symm3}
A_{\tau,\,\sigma}=\frac{1}{2}+\sum_{i+j\,\text{is}\,\text{odd}\atop
0\leq i,j\in \mathbb{Z}}\omega_{ij}
P_i(\tau)P_j(\sigma),\quad\omega_{ij}\in \mathbb{R}.
\end{equation}
\end{thm}
\begin{proof}
The result can be easily verified by using the same technique shown
in Theorem \ref{symthm} (or cf. \cite{Tangs14cor}).
\end{proof}

Theorem \ref{symmthm} is very useful for constructing symmetric csRK
methods in conjunction with Theorem \ref{mainthm}. It is easy to get
a symmetric RK methods based on symmetric csRK methods by using a
symmetric quadrature formula \cite{Tangs14cor}.

\subsection{Energy-preserving csRK methods}

Energy-preserving csRK methods were firstly studied in
\cite{Iavernarop07sst,quispelm08anc,brugnanoit10hbv,hairer10epv,tangs12ana,Tangs14cor},
and it was shown that there exists energy-preserving csRK methods
which are conjugate-symplectic up to a finite order
\cite{hairer10epv,hairerz13oco,Tangs14cor}. Miyatake
\cite{miyatake14aep} provided a sufficient condition for energy
conservation, and then he \& Butcher provided a proof for the
necessity of the condition in a ``weak" sense \cite{miyatake15aco}.
\begin{thm}\label{epcon}\cite{miyatake15aco}
A csRK method is energy-preserving if $\frac{\partial}{\partial
\tau}A_{\tau,\,\sigma}$ is symmetric, i.e.,
\begin{equation*}
\frac{\partial}{\partial
\tau}A_{\tau,\,\sigma}\equiv\frac{\partial}{\partial
\sigma}A_{\sigma,\,\tau},\; \text{for}\;\; \tau,\,\sigma \in [0,1],
\end{equation*}
and $A_{0,\,\sigma}\equiv 0,\; A_{1,\,\sigma}\equiv B_{\sigma}$.
\end{thm}
\begin{thm}\cite{tangs12ana,Tangs14cor}\label{EPm2}
Consider the csRK method \eqref{crk} with $B_\tau=1$, $C_\tau=\tau$
and
\begin{equation}\label{epmcoe2}
A_{\tau,\,\sigma}=\sum_{0\leq\iota\in\mathbb{Z}}\omega_{\iota}\int_0^{\tau}g_\iota(x)\,\dif
x\;g_\iota(\sigma),\;\omega_{\iota}\in\mathbb{R},
\end{equation}
where $g_\iota(x)\in L^2([0,1])$ with $g_\iota(x)=\sum\limits_{0\leq
\kappa\in\mathbb{Z}}a_{\iota \kappa}P_\kappa(x)$ (Legendre
expansion), $a_{\iota \kappa}\in\mathbb{R}$, then we have
\begin{itemize}
\item[\emph{(a)}]  $\breve{C}(\eta)$ holds if and only if the
parameters $\omega_{\iota}$ and $a_{\iota\kappa}$
($\iota,\kappa=0,1,2,\cdots$)  satisfy
$$\sum\limits_{0\leq\iota\in\mathbb{Z}}\omega_{\iota}a_{\iota i}a_{\iota
j}=
\begin{cases}
\delta_{ij},  &   0\leq i,j\leq \eta-1,\\
0, & 0\leq i \leq\eta-1,  j\geq\eta;
\end{cases}$$
\item[\emph{(b)}] if $\breve{C}(\eta)$ holds, then $\breve{D}(\eta-1)$
also holds;
\item[\emph{(c)}] the method is of order $2\eta_{_{M}}$, where
$\eta_{_{M}}=\max\{\eta\in \mathbb{Z}:
\breve{C}(\eta)\;\mbox{holds}\}$, and exactly preserves the energy
of system \eqref{Hs}.
\end{itemize}
\end{thm}

Actually, the condition shown in Theorem \ref{epcon} is essentially
equivalent to the formula\footnote{Alternatively, please refer to
our earlier work \cite{tangs12ana} which was presented as a report
during the international conference ``ICNAAM2012".} \eqref{epmcoe2},
since we can recast it as a series in terms of Legendre polynomials.
Some existing energy-preserving integrators (e.g., AVF methods
\cite{quispelm08anc}, $\infty$-HBVMs \cite{brugnanoit10hbv}, EPCMs
\cite{hairer10epv}, Galerkin time finite element methods
\cite{Tangs12tfe,Tangsc17dgm} etc) can be transformed into the csRK
methods described in Theorem \ref{EPm2}, and all of them possess an
even order. The following result says that there exists
energy-preserving B-series integrators which are
conjugate-symplectic up to a finite order (higher than their
algorithm order).
\begin{thm}\cite{Tangs14cor}\label{EPm1}
Apply the csRK method \eqref{crk} with $B_\tau=1$, $C_\tau=\tau$ and
\begin{equation}\label{epmcoe1}
A_{\tau,\,\sigma}=\sum_{0\leq\iota\in\mathbb{Z}}\omega_{\iota}\int_0^{\tau}P_\iota(x)\,\dif
x\,P_\iota(\sigma),\;\;\omega_0\equiv1,\;\omega_{\iota}\in\mathbb{R}
\end{equation}
to Hamiltonian system \eqref{Hs}, where $P_\iota(x)$ is the
$\iota$-degree Legendre polynomial. Assume
$\kappa:=\min\{\iota\in\mathbf{Z}:\;\omega_{\iota}\neq1\}<\infty$,
then the method is of order $2\kappa$, symmetric, energy-preserving
and conjugate-symplectic up to order at least $2\kappa+2$. If we
additionally require
$\frac{\omega_{\kappa}}{2\kappa-1}-\frac{\omega_{\kappa+1}}{2\kappa+1}
=\frac{2}{4\kappa^2-1}$, then the method is conjugate-symplectic up
to order $2\kappa+4$.
\end{thm}
\begin{rem}
If $\kappa=\min\{\iota\in\mathbf{Z}:\;\omega_{\iota}\neq1\}<\infty$
goes to $\infty$, then the energy-preserving csRK method formally
approximates to a conjugate-symplectic method (namely up to order
$\infty$). We tend to conjecture that within the framework of csRK
methods there exists no computational energy-preserving methods
which are conjugate to a symplectic method, though it needs to be
further investigated.
\end{rem}

%%%%%%%%%%%%%%%%%%%%%%%%%%%%%%%%%%%%%%%%%%%%%%%%%%%%%%%%%%%%%%%%%%%%%%%%%%%%%%%%%%%%%%%%%%%%%%%%%%%%%%%%%%
\section{Concluding remarks}

This note investigates the construction theory of RK-type methods
based on the recently-developed framework of RK methods with
``infinitely many stages". In the construction of RK-type
algorithms, a crucial technique associated with orthogonal
polynomial expansion is fully utilized. By using this approach, we
do not need to study the tedious solution of multi-variable
nonlinear algebraic equations stemming from order conditions. We
develop two ways to construct RK-type methods of arbitrarily high
order. As an important application for these theory, we study and
discuss the geometric numerical integration of Hamiltonian systems
by csRK methods. A sufficient algebraic condition for csRK methods
to be symplectic (resp. symmetric) is presented which is very
similar to the classic result. The necessity of these conditions
will be investigated elsewhere.

%%%%%%%%%%%%%%%%%%%%%%%%%%%%%%%%%%%%%%%%%%%%%%%%%%%%%%%%%%%%%%%%%%%%%%%%%%%%%%%%%%%%%%%%%%%%%%%%%%%%%%%%%%
\section*{Acknowledgments}

This work was supported by the National Natural Science Foundation
of China (11401055), China Scholarship Council and Scientific
Research Fund of Hunan Provincial Education Department (15C0028). We
are particularly grateful to J.C.~Butcher for pointing out that the
true origin of continuous-stage Runge-Kutta methods is from his
paper ``An algebraic theory of integration methods" published in
1972. And we claim that the descriptions about the origin in our
earlier paper \cite{Tangs14cor} is not accurate.

%%%%%%%%%%%%%%%%%%%%%%%%%%%%%%%%%%%%%%%%%%%%%%%%%%%%%%%%%%%%%%%%%

%% References


%%    \bibliographystyle{model1-num-names}
%%    \bibliography{<your-bib-database>}



%%%%%%%%%%%%%%%%%%%%%%%%%%%%%%%%%%%%%%%%%%%%%%%%%%%%%%%%%%%%%%%%

\end{document}